\documentclass{gtpart}   
\usepackage{pinlabel}
\usepackage{amsfonts}
\usepackage{amssymb}
\usepackage{amsmath}   
\usepackage{amsthm}
\usepackage{mathabx}
\usepackage{tikz}
\usepackage[all]{xy}

\title{The Farrell-Jones conjecture for Graph Products}

\author{Giovanni Gandini} 
\givenname{Giovanni}
\surname{Gandini}
\address{Institut for Matematiske Fag\\K{\o}benhavns Universitet\\\newline
         Universitetsparken 5, 2100 K{\o}benhavn \O\\Denmark}
\email{ggandini@math.ku.dk}
\urladdr{http://www.math.ku.dk/~zjb179}

\author{Henrik R\"uping} 
\givenname{Henrik}
\surname{R\"uping}
\address{Mathematisches Institut\\ Rheinische Wilhelms-Universit\"{a}t Bonn\\\newline
 Endenicher Allee 60, 53115 Bonn\\ Germany}
\email{henrik.rueping@hcm.uni-bonn.de}
\urladdr{http://www.math.uni-bonn.de/people/rueping}

\keyword{example}
\keyword{sample layout}
\subject{primary}{msc2010}{18F25}
\subject{secondary}{msc2010}{19A31}
\subject{secondary}{msc2010}{19B28}
\subject{secondary}{msc2010}{19G24}

\arxivreference{http://arxiv.org/abs/1211.6378}  
\arxivpassword{}   

\newtheorem{thm}{Theorem}[section]

\theoremstyle{definition}
\newtheorem{definition}[thm]{Definition}
\newtheorem{example}[thm]{Example}
\newtheorem{lemma}[thm]{Lemma} 
\newtheorem{prop}[thm]{Proposition}
\newtheorem*{rem}{Remark}

\newcommand{\ignore}[1]{}
\DeclareMathOperator {\IZ}{\mathbb{Z}}
  \DeclareMathOperator{\colim}{colim}
  \DeclareMathOperator {\calg}{\mathcal{G}}
  \DeclareMathOperator {\calc}{\mathcal{C}}
  \DeclareMathOperator {\calh}{\mathcal{H}}
  \DeclareMathOperator {\Aut}{Aut}
\makeop{Homo}
\numberwithin{equation}{section}

\begin{document}




\begin{abstract} We show that the class of groups satisfying the K- and L-theoretic  Farrell-Jones conjecture  is closed under taking  graph products of groups. 

\end{abstract}

\maketitle


A group $G$ satisfies  the \emph{K-theoretic Farrell-Jones conjecture with coefficients in  additive categories}  if for any additive $G$-category $\mathcal{A}$ the assembly map (induced by the projection $E_{\mathcal{VC}yc}G \to \mbox{pt}$)
$$\mbox{asmb}_n^{G, \mathcal{A}} : H_n^G(E_{\mathcal{VC}yc}G; \bold{K_{\mathcal{A}} })\to  H_n^G(\mbox{pt}; \bold{K_{\mathcal{A}} })$$
 is an isomorphism.
 There is an analogous L-theoretic version of the Farrell-Jones conjecture where the K-theory of $\mathcal{A}$ is replaced by the L-theory of $\mathcal{A}$ with  decoration $\langle-\infty\rangle$. We say that a group $G$ satisfies the $\mathbf{FJC}$ if  for any finite group $F$ the group  $G\wr F$ satisfies the K- and L-theoretic Farrell-Jones conjecture with coefficients in any additive category.
 The Farrell-Jones conjecture is very powerful, in fact it implies many famous conjectures; for example the Bass, the Borel, the Kaplansky and the Novikov conjecture \cite{BC}. The Farrell-Jones conjecture for K-theory up to dimension one implies that  the reduced projective class group $\Tilde{K}_0(RG)$ vanishes  whenever $R$ is a principal ideal domain and $G$ is torsion-free. When $R$ is the ring of integers the vanishing of  $\Tilde{K}_0(\mathbb{Z} G)$ implies Serre's conjecture, which claims that every group of type $\mbox{FP}$ is of type $\mbox{FL}$. \\ 
 Given a simplicial graph $\Gamma$ and a family of groups $\mathfrak{G} = \{G_v | v \in V \Gamma\}$, the \emph{graph product} $\Gamma\mathfrak{G}$ is the quotient of  free product $\bigast_{v\in V\Gamma} G_v $ by the relations
$[g_u,g_v]=1$ for all $g_u \in G_u$, $g_v \in G_v $ whenever $\{u, v\}\in E\Gamma$. Basic examples are given by right-angled Artin and right-angled Coxeter groups. \\
 Let $\calc$ be the class of groups satisfying the $\mathbf{FJC}$. The purpose of this note is to prove the following closure property for the class $\calc$.
 
\begin{thm}\label{main} The class $\calc$ is closed under taking graph products.
\end{thm}
Since right-angled Artin and right-angled Coxeter groups are CAT(0)-groups they satisfy $\mathbf{FJC}$ by  \cite[Theorem~B]{bartels2012borel}.

\subsubsection*{Acknowledgements}We want to thank Yago Antol\'in and Sebastian Meinert for their comments. This work was financed by the Leibniz-Award of Prof. Dr. Wolfgang L\"uck granted by the Deutsche Forschungsgemeinschaft.

\section{Preliminares }
Let us start by recalling some recent results.
\begin{prop}[\cite{bartels2008k, bartels2011farrell, bartels2012borel, bartels2012k}]\label{lem:FJinh} The class of groups $\calc$ has the following properties:
\begin{enumerate}
\item \label{lem:FJinh:hyp}Word-hyperbolic groups belong to $\calc$.
\item \label{lem:FJinh:CAT}$\ignore{}CAT$(0)-groups belong to $\calc$.
\item \label{lem:FJinh:poly} Virtually polycyclic groups belong to $\calc$.
\item \label{lem:FJinh:fext}The class $\calc$ is closed under taking finite index overgroups.
\item \label{lem:FJinh:sub}The class $\calc$ is closed under taking subgroups.
\item \label{lem:FJinh:colim} The class $\calc$ is closed under taking directed colimits.
\item \label{lem:FJinh:prod}The class $\calc$ is closed under taking direct sums.
\item \label{lem:FJinh:hom} Given any group homomorphism $f:G\rightarrow H$ and assume that $H\in \calc$ and $f^{-1}(Z)\in \calc$ for any torsion-free cyclic subgroup $Z$ of $H$. Then, $G$ belongs to  $\calc$.
\item \label{lem:FJinh:fprod}The class $\calc$ is closed under taking  free products.
\end{enumerate}
\end{prop}
\begin{proof}
The K-theoretic version without wreath products for hyperbolic groups was proved in \cite[Main~Theorem]{bartels2008k}. The L-theoretic version without wreath products for hyperbolic groups was proved in \cite[Theorem~B]{bartels2012borel}. This generalizes to wreath products of hyperbolic groups with finite groups as explained in \cite[Remark~6.4]{bartels2012k}.\\
First note that a wreath product of a $\ignore{}CAT(0)$ group and a finite group also acts geometrically on a $CAT(0)$ space. Thus it suffices to consider the setting without wreath product which was proved in \cite[Theorem~B]{bartels2012borel}.\\
A wreath product of a virtually polycyclic group with a finite group is again virtually polycyclic.
In  \cite[Theorem~0.1]{bartels2011farrell} it was proved that virtually polycyclic groups satisfy Farrell-Jones.\\
By \cite[Remark~0.5]{bartels2011farrell} the Farrell-Jones conjecture with finite wreath products passes to finite index overgroups.\\
The remaining properties  can be found in \cite[Theorems~1.7-1.10]{bartels2011farrell} without wreath products. 
Note that if $H$ is a subgroup of $G$ then $H \wr F\leq G\wr F$, and hence it satisfies $\mathbf{FJC}$.\\
We have for a direct system of groups $(G_i)_{i\in I}$ and a finite group $F$ that 
$(\colim_{i\in I}G_i) \wr F$ is isomorphic to $\colim_{i\in I}(G_i\wr F)$
and hence the $\mathbf{FJC}$ is also compatible with colimits. \\
Note that  the wreath product of a direct sum by a finite group $F$ embeds in the direct sum of wreath  products of the factors by $F$, hence we obtain that the $\mathbf{FJC}$ passes to finite direct sums. Hence by the previous point we have that $\mathcal{C}$ is closed under taking arbitrary direct sums.\\
Let $f:G\rightarrow H$ be a group homomorphism such that $H$ and $f^{-1}(Z)$ satisfy the $\mathbf{FJC}$ for any torsion-free cyclic subgroup $Z$ of $H$. Let $V$ be a virtually cyclic subgroup of $H$. Then it has a finite index torsion-free subgroup $Z$. Thus $f^{-1}(Z)$ has finite index in $f^{-1}(V)$. So $f^{-1}(V)$ also satisfies $\mathbf{FJC}$ by \eqref{lem:FJinh:fext}.\\
The same argument given in \cite[Lemma 3.16]{2009ph} for the fibered version of the Farrell-Jones conjecture applies to the $\mathbf{FJC}$.\\
Consider the surjection $\pi : G_1* G_2 \to G_1 \times G_2$ and a torsion-free cyclic subgroup $C$ of $G_1 \times G_2$.  Note that by \eqref{lem:FJinh:prod} $G_1 \times G_2$ satisfies $\mathbf{FJC}$ and the subgroup $\pi^{-1}(C)$ acts on the Bass-Serre tree of $G_1\ast G_2$ with trivial edge stabilizers. Since for all $g\in G_1\ast G_2$,  $G_i^g\cap \pi^{-1}(C)$ is either trivial or infinite cyclic, we have that $\pi^{-1}(C)$ is a free group and hence  it satisfy $\mathbf{FJC}$ by \eqref{lem:FJinh:hyp}. 
 \end{proof}
 In order to prove Theorem \ref{main} we need to show that the fundamental groups of certain graphs of groups satisfy $\mathbf{FJC}$.
\section{Some special graphs of groups}

Basic facts about groups acting on trees and graphs of groups can be found in \cite{serre1980trees}{}. Let us briefly recall how it is possible to recover a presentation for the fundamental group of a graph of groups $\calg =(X,G_?,i_*:G_*\hookrightarrow G_{t(*)})$. Pick a maximal tree $T\subset X$ and pick an orientation of each edge that is not in $T$. Let $L$ denote the set of those edges. Let $F(L)$ be the free group generated by $L$ and let $t_e$  be the generator corresponding to an edge $e\in L$. Then $\pi_1(\calg)$ is defined to be
$\bigast_{x\in V(X)} G_x * F(L)$ modulo the relations
\[ i_e(g)=i_{\overline{e}}(g) \mbox{ for }e\in T, g\in G_e \qquad t_ei_e(g)t_e^{-1} = i_{\overline{e}}(g) \mbox{ for }e\notin T, g \in G_e.\]
We have a canonical map $G_v\rightarrow \pi_1(\calg)$ for each vertex $v$. These  turn out to be injective. Furthermore we can also get canonical maps $G_e\hookrightarrow \pi_1(\calg)$. We view $G_e$ as a subgroup of $G_{t(e)}\subset \pi_1(\calg)$ if either $e\in T$ or $e\notin T$ and $e$ has the chosen orientation.

Graphs of groups arise as quotients of groups acting on trees. The edge and vertex groups are given by the stabilizers. Every graph of groups arises uniquely in  this way \cite[Section I.5.3]{serre1980trees}. The acting group is isomorphic to the fundamental group of the induced graph of groups. This is the statement of the fundamental theorem for groups acting on trees \cite[Theorem~13, Section~I.5.4]{serre1980trees}. Thus any subgroup of $\pi_1(\calg)$ inherits a decomposition as a fundamental group of a graph of groups.

\begin{definition}\label{bala} Let $\calg$ be a countable graph of groups where every vertex group is either trivial or infinite cyclic.  Since the edge groups inject into the vertex groups they are also either trivial or infinite cyclic. For every infinite stabilizer $G_x$  of a vertex $x$ pick a   generator $z_x$. Similarly, for every edge $e$ with non trivial stabilizer $G_e$ let $z_e$ be the generator of $G_e$. For such an edge  $e$ let $n_e\in\IZ$ be defined by $i_e(z_e)=z_{t(e)}^{n_e}$.\\
A closed edge loop $e_1,\ldots,e_m$  where each edge group is infinite cyclic is said to be \emph{balanced}, if 
\[\prod_{i=1}^m n_{e_i} = \prod_{i=1}^mn_{\overline{e_i}}.\]
Note that picking different generators does not change balancedness since every minus-sign appears twice.
The graph of groups is said to be \emph{balanced} if every closed edge loop with infinite cyclic edge stabilizers is balanced.
\end{definition}
\begin{example} The Baumslag-Solitar group $BS(m, n)=\langle a,b \,|\, b^{-1}a^{m}b =a^{n}\rangle$ is the fundamental group of a balanced graph of groups if and only if $m=n$.
\end{example}
\begin{lemma}\label{lem:balmapstohyp}
Let $\calg$ be a finite balanced graph of groups with infinite cyclic vertex and edge groups.  Then there is a central infinite cyclic subgroup $N\leq \pi_1(\calg)$ such that $\pi_1(\calg)/N$ is isomorphic to the fundamental group of a finite graph of finite groups.
\end{lemma}
\begin{proof}
 For a vertex $v$ pick  a generator $z_v$ of $G_v\leq \pi_1(\calg)$ and let  $n:=\prod_{e\in E(\calg)} n_e$. Let $N:=\langle z^n_v\rangle\leq \pi_1(\calg)$ for some vertex $v$. Let us first show that $z_v^n$ is contained in every vertex group. Let $v'$ be any vertex and let $e_1,\ldots,e_m$ be a path in $T$ from $v$ to $v'$. Now let us use the relations corresponding to those edges.  
 Note that from  \cite{serre1980trees} we have $z_v^n= z_{e_1}^{\frac{n}{n_{\overline{e_1}}}}= z_{t(e_1)}^{\frac{n \cdot n_{e_1}}{n_{\overline{e_1}}}} $. 
Now let $\overline{n_{i}}:=\frac{n\cdot n_{e_1}\cdot n_{e_2}\cdot\ldots \cdot n_{e_i}}{n_{\overline{e_1}}\cdot n_{\overline{e_2}}\cdot \ldots \cdot n_{\overline{e_i}}} $, by the definition of $n$ we have that  $\overline{n_{i}}$ is a non-zero integer and:
$$z^n_v = z^{\overline{n_{1}}}_{t(e_1)} =z^{\overline{n_{2}}}_{t(e_2)}=\ldots = z^{\overline{n_{m}}}_{v'}.$$
Thus $z_v^n$ is a power of $z_{v'}$ and so both elements commute.

Next we have to show that $z_v^n$ commutes with all elements of the form $t_e$ for some $e\in L$. Let $e_1,\ldots, e_m$ be a path in $T$ from $v$ to $t(e)$ and let $e'_1,\ldots, e'_{m'}$ be a path in $T$ from $v$ to $t(\overline{e})$. Let us abbreviate  
\[\overline{n_i}:=\frac{n\cdot n_{e_1}n_{e_2}\ldots n_{e_i}}{n_{\overline{e_1}}n_{\overline{e_2}}\ldots n_{\overline{e_i}}} \mbox{ and }
\overline{n'_i}:=\frac{n\cdot n_{e'_1}n_{e'_2}\ldots n_{e'_i}}{n_{\overline{e'_1}}n_{\overline{e'_2}}\ldots n_{\overline{e'_i}}}\]
as before. The path $e'_1,\ldots,e'_{m'},e,\overline{e_{m}},\ldots,\overline{e_1}$ is closed and the balancedness condition gives:
\[\left(\prod_{i=1}^{m'}n_{e'_i}\right)\cdot n_e\cdot \left(\prod_{i=1}^{m}n_{\overline{e_i}}\right)=\left(\prod_{i=1}^{m'}n_{\overline{e_{i}'}}\right)\cdot n_{\overline{e}}\cdot\left(\prod_{i=1}^{m}n_{e_i}\right).\]
With the abbreviations from above this gives
\begin{equation}n_e\overline{n'_{m'}}=n_{\overline{e}}\overline{n_m}\label{eq:balcond}\end{equation}
Analogously to the previous case we get 
\begin{equation}z_{t(e)}^{\overline{n_m}}=z_v^n=z_{t(\overline{e})}^{\overline{n'_{m'}}}.\label{eq:patheq}\end{equation}
Let us conjugate the first part by $t_e$ and use the relation
\begin{equation}t_e(z_{t(e)}^{n_e})t_e^{-1} = z_{t(\overline{e})}^{n_{\overline{e}}}\label{eq:fundrel}\end{equation}
which corresponds to $e$ to get:
\begin{eqnarray*}
t_ez_v^nt_e^{-1}&=&t_ez_{t(e)}^{\overline{n_m}}t_e^{-1}
=\left(t_ez_{t(e)}^{n_e}t_e^{-1}\right)^\frac{\overline{n_m}}{n_e}
=\left(z_{t(\overline{e})}^{n_{\overline{e}}}\right)^\frac{\overline{n_m}}{n_e}
=z_{t(\overline{e})}^{\overline{n'_{m'}}}
=z_v^n. \end{eqnarray*}
In the first equality we used the left hand side of \eqref{eq:patheq}, in the third equality \eqref{eq:fundrel}, the fourth equality uses \eqref{eq:balcond} and the last equality uses the right-hand side of \eqref{eq:patheq}.\\
Thus $z_v^n$ commutes also with all generators of the form $t_e$ and hence $N$ is a central subgroup of $\pi_1(\calg)$. Furthermore the argument from above showed that $N$ is contained in any edge group.

Now let us pass to the actions on trees. Let $T'$ be the universal cover of $\calg$. Let us show that $N$ acts trivially. Let $e\in T'$ be some edge. Its stabilizer is some conjugate of some edge group $G_e\subset \pi_1(\calg)$, say $gG_eg^{-1}$ for some $g\in \pi_1(\calg)$. Since $N$ is a central subgroup of $\pi_1(\calg)$ that is contained in each edge group $G_e\subset \pi_1(\calg)$ we have $N=gNg^{-1}\subset gG_eg^{-1}$. So the action factorizes as $\pi_1(\calg)\rightarrow \pi_1(\calg)/N\rightarrow \Aut(T')$. Especially note that the underlying graphs of $\pi_1(\calg)\backslash T'$ and $\left(\pi_1(\calg)/N\right)\backslash T'$ are the same. Thus $\pi_1(\calg)/N$ can also be expressed as a fundamental group of a finite graph of groups. The vertex and edge groups are $G_x/N$. These are quotients of infinite cyclic groups by a nontrivial subgroup. Thus they are finite.
\end{proof}

\begin{lemma}\label{lem:FJbal} Let $\calg$ be a balanced graph of groups. Then $\pi_1(\calg)$ satisfies the $\mathbf{FJC}$.
\end{lemma}
\begin{proof}
For the sake of simplicity we would like to define the fundamental group of a disconnected graph of groups as an assignment that assigns to any connected component the fundamental group of its graph of groups. First note that we can exhaust any graph by its finite subgraphs. Then its fundamental group will be the colimit of the fundamental group of the finite subgraphs of groups. So by Proposition~\ref{lem:FJinh}(\ref{lem:FJinh:colim}) it suffices to consider a finite graph of groups with these properties. 
\begin{figure}[h!]
\centering
\begin{tikzpicture}[scale=0.95] 
\draw (0,0)node {$\bullet$} .. controls (-1,0) and (0,-1) .. (0,0);
\draw (0,0)--(1,0)node {$\bullet$}--(0.5, 0.75) node {$\bullet$}-- (0,0); 
\draw [dotted](1,0)--(0.9,-0.75)node {$\bullet$};  
\draw [dotted](0.5,0.75)--(1.5,1)node {$\bullet$}--(2.25, 1.5);  
\draw (2.25,1.5)node {$\bullet$}--(2.75, 2.25)node {$\bullet$}--(3.25,1.5)node {$\bullet$};
\draw[dotted](2.25,1.5)--(3.25,1.5);
\draw [dotted](1.5,1) .. controls (2,0.625) and (2,0.625) .. (2,0);
\draw [dotted](1.5,1) .. controls (1.5,0.375) and (1.5,0.375) .. (2,0);
\draw (2,0)node {$\bullet$};
\draw (2,0) .. controls (1,0) and (2,-1) .. (2,0);
\draw[dotted](2, 0)--(2.75,-0.75);
\draw (2.75,-0.75)node {$\bullet$} .. controls (3.75,-0.75) and (2.75,-1.75) .. (2.75,-0.75);
\end{tikzpicture}
\caption{The dotted lines correspond to edges with trivial stabilizers and the lines correspond to edges with infinite cyclic stabilizers.}
\label{figure:fig}
\end{figure}
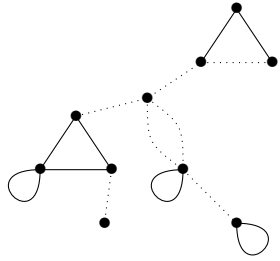
Next consider the subgraph consisting of all edges with nontrivial stabilizers and all vertices as illustrated in Figure \ref{figure:fig}. A connected component  of a graph is a maximal, connected subgraph.
 Let us show that the fundamental group of each of its connected components satisfies the $\mathbf{FJC}$. Note that an edge with nontrivial stabilizer can only connect vertices with nontrivial stabilizer. So the connected component of vertices with trivial stabilizers consists only of points and hence their fundamental group is trivial.\\
Let us now consider a connected component $\calc$ of a vertex $v$ with nontrivial stabilizer. By Lemma~\ref{lem:balmapstohyp} we find an infinite cyclic subgroup $N\leq \pi_1(\calc)$ such that $\pi_1(\calc)/N$ is fundamental group of a finite graph of finite groups and hence virtually-free. 
So $\pi_1(\calc)/N$ satisfies $\mathbf{FJC}$ by Proposition~\ref{lem:FJinh}(\ref{lem:FJinh:hyp}). Let us apply Proposition~\ref{lem:FJinh}(\ref{lem:FJinh:hom}) to the map $f:\pi_1(\calc)\rightarrow \pi_1(\calc)/N$. Since the kernel of $f$ is infinite cyclic any preimage $f^{-1}(Z)$ of an infinite cyclic subgroup $Z$ of $\pi_1(\calc)/N$ will satisfy the $\mathbf{FJC}$ by being virtually abelian.\\
Next we have to add edges to get the full graph. If we add an edge we either merge two connected components or add an edge to a single connect component. In the first case the fundamental group of the new component is the free product of the fundamental groups of the previous components. In the second case the new fundamental group is a free product of $\IZ$ with the old fundamental group. In both cases the fundamental group of the new component satisfies the $\mathbf{FJC}$ by Proposition~\ref{lem:FJinh}(\ref{lem:FJinh:fprod}). This completes the proof.
\end{proof}

The following criterion implies balancedness.

\begin{lemma}\label{lem:critBal} Let $\calg$ be a graph of groups where every vertex group is either trivial or infinite cyclic. Assume that for each edge $e$ with infinite cyclic stabilizer there exists a homomorphism $f_e:\pi_1(\calg)\rightarrow \IZ$ such that $f_{e}|_{G_e}$ is nontrivial. Then $\calg$ is balanced.
\end{lemma}
\begin{proof}
Let us again use the notation of Definition~\ref{bala}. Recall that an egde $e'$ gives a relation in the fundamental group. If the edge is contained in the chosen maximal tree, the relation will be $z_{t(\overline{e'})}^{n_{\overline{e'}}}=z_{e'}=z_{t(e')}^{n_{e'}}$. Otherwise it will be $t_{e'}^{-1}z_{t(\overline{e'})}^{n_{\overline{e'}}}t_{e'}=z_{e'}=z_{t(e')}^{n_{e'}}$.
Since the target of $f_e$ is abelian we get in both cases $f_e(z_{t(\overline{e'})})^{n_{\overline{e'}}}= f_e(z_{t(e')})^{n_{e'}}$.\\
Let $e_1,\ldots, e_m$ be a loop of edges with infinite stabilizers, and let 
\[n:= \prod_{i=1}^m n_{e_i},\qquad \overline{n} :=\prod_{i=1}^m n_{\overline{e_i}}.\]
So we get: 
\begin{multline*}f_{e_m}(z_{t(e_m)})^{\overline{n}}=f_{e_m}(z_{t(\overline{e_1})})^{\overline{n}}=f_{e_m}(z_{t(e_1)})^\frac{\overline{n}n_{e_1}}{\overline{n_{e_1}}}= \\
f_{e_m}(z_{t(\overline{e_2})})^\frac{\overline{n}n_{e_1}}{\overline{n_{e_1}}}=f_{e_m}(z_{t(e_2)})^{\frac{\overline{n}n_{e_1}n_{e_2}}{\overline{n_{e_1}}\overline{n_{e_2}}}}=\ldots =f_{e_m}(z_{t(e_m)})^n.\end{multline*}

Since $0\neq f_{e_m}(z_{e_m})=f_{e_m}(z_{t(e_m)})^{n_{e_m}}$ this implies that $n=\overline{n}$.
So the circle $e_1,\ldots,e_m$ is balanced and since the circle was chosen arbitrarily, the graph of groups is balanced.
\end{proof} 

\begin{lemma}\label{lem:AutFJProd} Given two groups $G_1,G_2$ with corresponding automorphisms $\varphi_1,\varphi_2$. Suppose that the $\mathbf{FJC}$ holds for $G_1\rtimes_{\varphi_1} \IZ$ and for $G_2\rtimes_{\varphi_2} \IZ$. Then it also holds for $(G_1\times G_2)\rtimes_{\varphi_1 \times\varphi_2}\IZ$.
\end{lemma}
\begin{proof} Projection to the first factor induces a homomorphism
\[f:(G_1\times G_2)\rtimes_{\varphi_1\times \varphi_2}\IZ \rightarrow G_1\rtimes_{\varphi_1} \IZ.\]
We want to use Proposition~\ref{lem:FJinh}(\ref{lem:FJinh:hom}), the target satisfies the Farrell-Jones conjecture with wreath products by assumption. The kernel is isomorphic to $G_2$ and hence it also satisfies  the $\mathbf{FJC}$ by Proposition~\ref{lem:FJinh}(\ref{lem:FJinh:sub}). Let $C$ be an infinite cyclic subgroup of $G_1\rtimes_{\varphi_1} \IZ$ and let $(g,m)$ be a generator. Then $f^{-1}(C)$ is isomorphic to $G_2\rtimes_{\varphi_2^m}\IZ$. For $m=0$ the preimage  $f^{-1}(C)$ is just a direct product of two groups satisfying $\mathbf{FJC}$. Otherwise $f^{-1}(C)$ is isomorphic to the subgroup $G_2\rtimes_{\varphi_2}m\IZ$ and again by Proposition~\ref{lem:FJinh}(\ref{lem:FJinh:sub}) it also satisfies the $\mathbf{FJC}$.
\end{proof}

\begin{lemma}\label{lem:AutFJFrProd} Given two groups $G_1,G_2$ with automorphisms $\varphi_1,\varphi_2$. Suppose that the $\mathbf{FJC}$ holds for $G_1\rtimes_{\varphi_1} \IZ$ and for $G_2\rtimes_{\varphi_2} \IZ$. Then it also holds for $(G_1* G_2)\rtimes_{\varphi_1*\varphi_2}\IZ$.
\end{lemma}
\begin{proof} We want to apply Proposition~\ref{lem:FJinh}(\ref{lem:FJinh:hom}) to the canonical homomorphism
\[f:(G_1* G_2)\rtimes_{\varphi_1*\varphi_2}\IZ \rightarrow (G_1\times G_2)\rtimes_{\varphi_1\times \varphi_2}\IZ.\]
The target satisfies the $\mathbf{FJC}$  by Lemma~\ref{lem:AutFJProd}. 
The group $G_1 * G_2$ acts on a tree $T$ with trivial edge stabilizers and vertex stabilizers conjugate to $G_1$ or $G_2$. We can identify the set of edges equivariantly with $G_1*G_2$ where the action is given by left multiplication. The set of vertices can be identified with $G_1*G_2/G_1\amalg G_1*G_2/G_2$  and the two endpoints of an edge are its cosets. We can extend the action  to the group $(G_1* G_2)\rtimes_{\varphi_1*\varphi_2}\IZ$ on the tree $T$ if we let the generator $t$ of $\IZ$ act in the following way:
\begin{eqnarray*} 
t\cdot g &:=& \varphi_1*\varphi_2(g) \mbox{ for an edge }g,\\
\quad t\cdot gG_i&:=&\varphi_1*\varphi_2(gG_i)=\varphi_1*\varphi_2(g)G_i\mbox{ for a vertex }gG_i.
\end{eqnarray*}

The quotient of the tree $T$ by the group $(G_1*G_2)\rtimes_{\varphi_1*\varphi_2} \IZ$ will be  a segment since the quotient of $T$ by $G_1 * G_2$ is a segment and $t$ does not swap the vertices. It is easy to see that this  gives the following graph of groups decomposition for $(G_1*G_2)\rtimes_{\varphi_1*\varphi_2} \IZ$: 
\begin{align*}
\xymatrixrowsep{0.03in}
\xymatrix{
\bullet\ar@{-}[r]^\IZ & \bullet \\
G_1\rtimes \IZ & G_2\rtimes\IZ}
\end{align*}

Now let $C$ be an infinite cyclic subgroup of $(G_1\times G_2)\rtimes_{\varphi_1\times \varphi_2}\IZ$.  The kernel of $f$ acts freely on $T$. So no nontrivial element from the kernel can stabilize some edge or some vertex. Restrict the  group action to $f^{-1}(C)$, then $f$ maps injectively  every  stabilizer into $C$.
So the group $f^{-1}(C)$ can be written as the fundamental group of a graph of groups $\calg$ where every vertex group and every edge group is either infinite cyclic or trivial. \\
Any infinite edge group is subconjugate to $(1, \IZ) \leq (G_1* G_2)\rtimes_{\varphi_1*\varphi_2}\IZ$. So if we restrict the projection $\pi :\IZ \leq (G_1* G_2)\rtimes_{\varphi_1*\varphi_2}\IZ\rightarrow \IZ$ to this edge group we get a nontrivial homomorphism. So $\calg$ is balanced by Lemma~\ref{lem:critBal}. Hence $f^{-1}(Z)=\pi_1(\calg)$ satisfies the Farrell-Jones conjecture by Lemma~\ref{lem:FJbal}. This completes the proof.
\end{proof}
 Let $S$ be a $\IZ$-set and let $G$ be a group. We define $G_S:=\left(\bigast_{s\in S} G\right)\rtimes \IZ$ where $\IZ$ acts on  $\bigast_{s\in S} G$ by 
$m\cdot g_s:=g_{ms}$. Here $g_s$ denotes the group element $g$ lying in the $s$-th copy of $G$. 
\begin{lemma}\label{lem:PermFactorsIsOK} For any $\IZ$-set $S$  the group  $G_S$ satisfies the $\mathbf{FJC}$ whenever $G$ does.
\end{lemma}
\begin{proof} Any $\IZ$-set is the disjoint union of its orbits. Furthermore, let $\overline{S}=\{S' \subset S : S' \mbox{ is a finite union of orbits} \}$, then we have: 
\[G_S=\left(\bigast_{s\in S} G\right)\rtimes \IZ = \colim_{S'\in \overline{S}}\left(\bigast_{s\in S'} G\right)\rtimes \IZ.\]
So by Proposition~\ref{lem:FJinh}(\ref{lem:FJinh:colim}) it suffices to consider the case where $S$ has only finitely many orbits. By Lemma~\ref{lem:AutFJFrProd}  it suffices to consider the case of a single orbit only. \\
If $S$ is of the form $\IZ/m$ then the finite index subgroup $\left(\bigast_{s\in \IZ/m} G\right)\rtimes m\IZ$ is isomorphic to a direct product 
$\left(\bigast_{s\in \IZ/m} G\right)\times \IZ$ and so it satisfies the Farrell-Jones conjecture by Proposition~\ref{lem:FJinh}(\ref{lem:FJinh:prod}) and Proposition~\ref{lem:FJinh}(\ref{lem:FJinh:fext}). 
If $S$ is of the form $\IZ$ we get $\left(\bigast_{s\in \IZ} G\right)\rtimes \IZ\cong G\ast \IZ$ and  it also satisfies the $\mathbf{FJC}$ by Proposition~\ref{lem:FJinh}(\ref{lem:FJinh:prod}).
\end{proof}
\begin{lemma}\label{permut} Let $\Gamma \mathfrak{G}$ be a graph product and  $v\in V\Gamma$. Consider the graph $\Gamma'$ obtained by removing $v$ from $\Gamma$ and let $\pi : \Gamma \mathfrak{G} \to \Gamma' \mathfrak{G}$ be the canonical surjection. Then, for every cyclic subgroup $C$ its preimage $\pi^{-1}(C)$ is isomorphic to $\left(\bigast_{\Gamma' \mathfrak{G}/ \Gamma'' \mathfrak{G}} G_v\right)\rtimes C$ where $\Gamma''$ denotes the full subgraph of $\Gamma'$ consisting of all neighbors of $v\in \Gamma$.The group $C$ acts on $\bigast_{\Gamma' \mathfrak{G}/ \Gamma'' \mathfrak{G}} G_v$ by permuting the free factors.
\end{lemma}
\begin{proof}By the argument given in the proof of \cite[Theorem 4.1]{HR} the kernel of $\pi$  is isomorphic to $\bigast_{\Gamma' \mathfrak{G}/ \Gamma'' \mathfrak{G}} G_v$ where the isomorphism is given by $g_{h\Gamma'' \mathfrak{G}}\mapsto g^h$.
The  preimage $\pi^{-1}(C)$ is isomorphic to $\left(\bigast_{\Gamma' \mathfrak{G}/ \Gamma'' \mathfrak{G}} G_v\right)\rtimes C$ where $C$ acts by conjugation on $\bigast_{\Gamma' \mathfrak{G}/ \Gamma'' \mathfrak{G}} G_v$.  Moreover, using the explicit isomorphism from above we get that $C$ permutes the free factors. 
\end{proof}
\begin{proof}(Theorem \ref{main})  Let  $v\in V\Gamma$ and let $\Gamma'$ and $\Gamma''$ be defined as above. We apply Proposition ~\ref{lem:FJinh}\eqref{lem:FJinh:hom} to the canonical surjection $\pi : \Gamma \mathfrak{G} \to \Gamma' \mathfrak{G}$. 
The result follows immediately from an application of Lemma~\ref{permut} and Lemma~\ref{lem:PermFactorsIsOK}.
\end{proof}
\begin{rem} Answering a long-standing open question Bestvina and Brady showed the existence of non-finitely presented groups of type $\mbox{FP}$ \cite{BB-97}.  Given a non-trivial right-angled Artin group $G$ associated to a flag complex $L$ there is a canonical surjection $\phi : G\twoheadrightarrow \mathbb{Z}$ given by sending each generator to $1$.  Bestvina and Brady proved that the finiteness properties of the kernel of $\phi$ are dictated by the connectivity property of $L$ and that   the kernel of $\phi$ is of type $\mbox{FP}$ if and only if is of type $\mbox{FL}$. Now consider a graph product $\Gamma\mathfrak{G}$ such that every vertex group $G_v$ surjects onto $\mathbb{Z}$.  These maps induce  a surjection $\phi_V : \Gamma\mathfrak{G} \twoheadrightarrow \mathbb{Z}$ and it would be interesting to study the finiteness property of the kernel of $\phi_V$. Note that whenever the $G_v$'s satisfy the $\mathbf{FJC}$ it follows from Theorem \ref{main} and Proposition~\ref{lem:FJinh}(\ref{lem:FJinh:sub}) that so does  $\ker{\phi_V}$. Hence $\ker{\phi_V}$ cannot solve in negative Serre's conjecture. However, there are many indicable groups for which the $\mathbf{FJC}$ is unknown, some examples are given by non-solvable Baumslag-Solitar groups (see \cite{farrell2013farrell} for the solvable ones).  
\end{rem}

\bibliographystyle{alpha}

\end{document}